\documentclass{amsart}
\usepackage{hyperref}

\newtheorem{theorem}{Theorem}[section]
\newtheorem{lemma}[theorem]{Lemma}

\theoremstyle{definition}
\newtheorem{definition}[theorem]{Definition}
\newtheorem{example}[theorem]{Example}

\theoremstyle{remark}

\usepackage{verbatim} 
\numberwithin{equation}{section}

\begin{document}
\title{An Approach with Toric Varieties for Singular Learning  Machines.}


\author{M.P. Castillo-Villalba}
\address{Program of Computational Genomics, Center for Genomic Sciences UNAM, C.P. 62210, Cuernavaca, Morelos, Mexico.}
\email{mpolovillalba@gmail.com}

\author{J.O. Gonz\'{a}lez -Cervantes.}
\address{Department of Mathematics, Superior School of Physic and Mathematics, National Polytechnical Institute, IPN, Zacatenco, CP 07738, Mexico City, Mexico.}
\email{jogc200678@gmail.com}

\subjclass[2010]{Primary}

\keywords{Toric variety; toric morphism; lattice polytope; Hilbert basis; learning curves; Singular machines, Kullback distance.}
\ams{AMS classification codes.}

\date{19/Jul/2017.}


\begin{abstract}
The Computational Algebraic Geometry applied in Algebraic Statistics; are beginning to exploring new branches and applications; in artificial intelligence and others areas. Currently, the development of the mathematics is very extensive and it is difficult to see the immediate application of few theorems in different areas, such as is the case of the Theorem 1 given in \cite{GEwald1993} and proved the middle here. Also this work has the intention to show the Hilbert basis as a powerful tool in data science; and for that reason we compile important results proved in the works by, S. Watanabe \cite{SWatanabe2009}, D. Cox, J. Little and H. Schenck \cite{DCoxLittleSchenck2010}, and G. Ewald \cite{GEwald1993}. In this work we study, first, the fundamental concepts in toric algebraic geometry. The principal contribution of this work is the application of Hilbert basis (as one realization of Theorem 1) for the resolution of singularities with toric varieties, and a background in lattice polytope. In the second part we apply this theorem to problems in statistical learning, principally in a recent area as is the Singular Learning Theory. We define the singular machines and the problem of \textbf{Singular Learning} through the computing of learning curves on these statistical machines. We review and compile results on the work of S. Watanabe in \textbf{Singular Learning Theory}, ref.; \cite{SWatanabe12001}, \cite{SWatanabe42001}, \cite{SWatanabe52001}, we formalize this theory with toric resolution morphism in a theorem proved here (Theorem 6), characterizing these singular machines as toric varieties, and we reproduce results previously published in Singular Statistical Learning in \cite{SWatanabe32001}, \cite{SWatanabe42001}, \cite{SWatanabe72001}.

\end{abstract}

\maketitle
\section{Preliminars.}
The paper is organized as follows. In the first part, we revise a few concepts of convex geometry, the Gordan lemma and separation lemma, as important preliminary results for the subsequent developments as Hilbert basis. In the second section we revise the standard theory of toric algebraic geometry, \cite{BSturmfels1996}, \cite{DCoxLittleSchenck2010}, \cite{GEwald1993}; and make use of the definition of toric variety as an algebraic affine scheme, a definition that will permit the formalizations we show for singular machines and S-systems. In the third section, we enunciate a proof of the Hilbert basis lemma and we compute toric ideals. The value of this result and the computing by means of the Singular program, ref., \cite{DGPS}, enable us to compute  toric ideals as the basis for applications in statistical learning. Furthermore, we define toric morphisms and gluing maps which are of the great importance for the proof of theorem 1. These applications give evidence of the relevance of theorem 1 and its potential benefit to facilitating solutions of problems in engineering.\\
We also give a formal definition of singularity, Ewald \cite{GEwald1993}, and enunciate two theorems for toric resolution; one of them is the theorem of Atiyah-Hironaka; S. Watanabe, \cite{SWatanabe12001}, which is applied for the resolution of singularities due to S. Watanabe, \cite{SWatanabe12001}, \cite{SWatanabe32001}, \cite{SWatanabe62001}, this fact is our motivation to study toric varieties in singular machines and embedding of its parameter space associated, into projective spaces as the theorem 3 proves.\\
In the fourth section, we study and summarize the main concepts of statistical singular learning (identifiable and non identifiable machines, Kullback distance, Fisher matrix information, learning curve and singular machines) with the purpose of making a formal study of singular machines by means of toric resolutions and affine toric varieties where we enunciate and prove part of the Theorem 6 applying the results of the first part. We also see the effect of the singularities in statistical learning and its importance for the performance and training in singular machines, \cite{SWatanabe32001}, which is resolved and studied by means of Theorem 6. We conclude this section with applications for three different statistical machines (perceptron of two layers, mix of binomial distributions, and three layer perceptron) and compute the learning curves by means of Hilbert basis reproducing the results of S. Watanabe, \cite{SWatanabe52001},\cite{SWatanabeYamazaki2002}, \cite{KYamazakiWatanabe2004}.

\section{Background of Convex Combinatorial Geometry.}
All this compilation of definitions and concepts can be consulted in; G. Ewald, \cite{GEwald1993}. 
A set   $ S \, \subset \, \mathbb{R}^{n}\setminus \emptyset $ is a  convex set if   each  $ \alpha \,\in  S $ is a convex  combination of elements of $S$; that is,   $\displaystyle \alpha= \sum_{i=1}^{r} \lambda_{i} \alpha_{i}$, where $ \lambda_{i} \geq \, 0  $  and $\alpha_i \in S$  for all $i =1,\dots, r$, with       $ \displaystyle \sum_{i=1}^{r} \lambda_{i}=1 $. \\
Given  $M \subset \mathbb{R}^{n} $, by \textbf{conv}$ M $ we mean the \textbf{hull convex} of $M$, which is the set of all convex combinations of elements of $M$. Moreover, if $ M $ is a finite set then \textbf{conv}$ M $ is called  a \textbf{convex polytope} or    \textbf{polytope}. \\
 A \textbf{lattice} $ N $ is a free abelian group of finite rank, and if its rank is $ n \in\mathbb N $, then $N$ is isomorphic to $ \mathbb{Z}^{n} $.\\
 Let  $ M $ and $ N $ be two lattice  both of rank $n$, consider $\:  \langle \cdot, \cdot \rangle  :  M \times N   \,  \longrightarrow \, \mathbb{Z} $,  the usual homomorphism of lattice from the inner product in $\mathbb R^n$ and identify to $ N $ with $ Hom_{\mathbb{Z}}(M,\mathbb{Z}) $, then we say 
that $ N $ is the dual lattice of the lattice $ M $, and reciprocally. In any case one denotes $ N=M^{\vee} $, see for more details of this formalism \cite{DCoxLittleSchenck2010}.\\
 Given  $ M $ and $ N $  as dual lattice, denote  $ M_{\mathbb{R}} = M\otimes_{Z}\mathbb{R}$ and $ N_{\mathbb{R}} = N\otimes_{Z}\mathbb{R} $, and  set  $ \sigma =Con( S  ) \subseteq \, M_{\mathbb{R}} $,  for some set $ S \, \subseteq \, M $, then $ \sigma $ is called a \textbf{rational polyhedric cone} or  \textbf{lattice cone}, \cite{DCoxLittleSchenck2010}.\\
 Also a \textbf{lattice cone} is a cone $ \sigma= Con(\alpha_{1},...,\alpha_{r}) \subset \, \mathbb{Z}^{n} $, generated by, $\alpha_{1},...,\alpha_{r}  \,  \in  \mathbb{Z}^{n}$ vectors. If the  coordinates of $\alpha_i$ are relative primes to pairs for each $i=1,\dots,r$, then $\alpha_{1},...,\alpha_{r} $ are called \textbf{primitive vectors}, and the cone $ \sigma $ is called a \textbf{regular cone}. It is well known that if $ \alpha_{1},...,\alpha_{r}  $, are primitives, then there exists $ \alpha_{r+1},...,\alpha_{n} \, \in \mathbb{Z}^{n}  $ such that:
\begin{center}
\textbf{Det}$ (\alpha_{1},...,\alpha_{n})= \pm 1 $.
\end{center}
Also, if the $ \alpha_{1},...,\alpha_{n} \, \in \mathbb{Z}^{n} $ are linearly independent, then the cone $ \sigma $ is a \textbf{simplex cone or simplicial cone.}\\
 A \textbf{face}, $ \tau $ of a cone $ \sigma $ is $ H_{p}\cap \sigma $ where $ H_{p} \subset \, \mathbb{R}^{n} $  is a  tangent hyperplane   to $\sigma$ at  $ p \in \sigma $, it is usually denoted by $ \tau \, \preceq \, \sigma $ it is well known $ \preceq $ is a relation of order.\\
 The   relative interior of $ \sigma $ is   $Relint(\sigma)= \sigma^{\vee}  \setminus    \sigma^{\perp}   $, where
\begin{center}
$\sigma^{\vee} = \lbrace m \in \sigma: \langle m,u \rangle > 0, \, \forall u \, \in \sigma \rbrace $,\\
$ \sigma^{\perp}= \lbrace m \in \mathbb{R}^{n}: \langle m,u \rangle =0, \, \forall u \, \in \sigma  \rbrace$.
\end{center}
  Let  $ P \, \subseteq \, M_{\mathbb{R}} $  be a lattice polytope. A set of cones $ \sum_{F} = \lbrace \sigma_{F} \, \vert \, F \, \preceq \,  P \rbrace$,  is called a \textbf{Fan} if and only if:
	\begin{itemize}
	\item  If  $\tau \preceq \sigma_{F}  $, then $ \tau \preceq \sum_{F}  $  for each  $\sigma_{F} \, \in \, \sum_{F} $,   .\\
\item  If $\tau = \sigma_{1}\cap \sigma_{2} $, then  $ \tau \, \preceq \, \sigma_{1} $ and $ \tau \, \preceq \, \sigma_{2} $ for each $
\sigma_{1}, \sigma_{2}  \, \in \,  \sum_{F}$.
 \end{itemize}
Recalling the following facts:
\begin{enumerate} 
\item Separation Lemma: If $ \sigma_{1} $ and $ \sigma_{2} $ are lattice cones in $ M $, whose intersection $ \tau =\sigma_{1} \cap \sigma_{2} $ is a face of both, then there is exists a hyperplane $ H_{m} $ such that:
\begin{center}
$ \tau =\sigma_{1} \cap H_{m}=\sigma_{2} \cap H_{m}$,
\end{center}
for any $ m \, \in \, Relint(\sigma)= \sigma_{1}^{\vee} \cap (-\sigma_{2})^{\vee}$.\\
\item Lemma. Set $\tau \, \preceq \, \sigma$ and $ m \,\in Relint(\tau^{\perp} \cap \sigma^{\vee} ) \setminus   \lbrace 0 \rbrace $. Then
\begin{center}
$ \tau^{\perp}= \sigma^{\vee} \oplus\{ \lambda (-m)\ \mid \  \lambda \, \in \, \mathbb{R} \}$.
\end{center}
\item Lemma. Let $ \sigma \, \subset \, \mathbb{R}^{n} $ be a lattice cone, then $ \sigma \cap \mathbb{Z}^{n} $ is a \textbf{monoid}.
\item Gordan lemma. Let $ \sigma \, \subset \, \mathbb{R}^{n} $ be a lattice cone,   then the monoid $ \sigma \cap \mathbb{Z}^{n} $ is  finitely generated.
\item Theorem. Let $ \sigma \subset \, \mathbb{R}^{n} $ be a  $ n $ dimensional  lattice cone   with \'{a}pex $ 0 $, i.e.; $0 \preceq \sigma $, and let $ b_{1},...,b_{r} $ be  the inner normal facets of $ \sigma $. Then
\begin{center}
$  \sigma^{\vee}= Con(b_{1},...,b_{r})$.
\end{center}
\end{enumerate}  

\subsection{About toric algebraic geometry.}

 The affine variety $(\mathbb{C}^{*})^{n} =(\mathbb{C}/ \lbrace 0 \rbrace)^{n} $ is a group equipped  with the  complex product of coordinates to pairs and it is  called the \textbf{complex algebraic  n-torus}. A \textbf{torus T} is an affine variety isomorphic to $(\mathbb{C}^{*})^{n}$.\\
 A \textbf{character} of a torus \textbf{T} is a homomorphism of groups, $ \chi  : T\longrightarrow \mathbb{C}^{*} $. For example,  set  $ m=(a_{1},...,a_{n}) \, \in \mathbb{Z}^{n} $, then $ \chi^m  :(\mathbb{C}^{*})^{n} \longrightarrow \mathbb{C}^{*} $   given by:
\begin{center}
$  \chi^{m}(t_{1},...,t_{n})=t_{1}^{a_{1}}*...*t_{n}^{a_{n}} $,
\end{center}
is \textbf{character} of $(\mathbb{C}^{*})^{n} $. Even more, it is well known that any character of $(\mathbb{C}^{*})^{n}$ is given  as above.  
Note that given  a lattice $ M $  and $m\in M$, then it is possible to define a character of $T$ by  $ \chi^{m}:T \longrightarrow \mathbb{C}^{*} $.\\

By an \textbf{uni-parametric subgroup} of a torus \textbf{T}  we mean  a homomorphism of groups $ \lambda : \mathbb{C}  ^{*} \longrightarrow  T$. Given  $ u =(b_{1},...,b_{n}) \, \in \mathbb{Z}^{n}$  define $ \lambda^{u}: \mathbb{C} ^{*} \longrightarrow  (\mathbb{C}^{n})^{*} $ by:
\begin{center}
$ \lambda^{u}(t_{1},...,t_{n})=(t_{1}^{b_{1}}, \dots, t_{n}^{b_{n}})$.
\end{center}
Then $\lambda^{u}$
 is  a uni-parametric subgroup of $(\mathbb{C}^{n})^{*}$ and any uni-parametric subgroup of $(\mathbb{C}^{n})^{*}$ is given in the same form.

One sees that  given a Torus \textbf{T},  there holds that   all uni-parametric subgroups of   $T$ form a free abelian group $ N $ with the same dimension of \textbf{T}. The same fact is obtained for all  characters of  \textbf{T}.


 The ring: 
\begin{center}
$  \mathbb{C}[t,t^{-1}]=\mathbb{C}[t_{1},...,t_{n},t_{1}^{-1},...,t_{n}^{-1}] $
\end{center}
is called the \textbf{ ring of Laurent polynomials} and the monomials,
\begin{center}
$  \lambda * t^{a}= \lambda *t^{a_{1}}*...*t^{a_{n}}$ with $ a=(a_{1},...,a_{n}) \, \in \mathbb{Z}^{n}, \, \lambda \, \mathbb{C^{*}}$.
\end{center}
are called  \textbf{Laurent monomials}.\\

The \textbf{support} of  a Laurent polynomial  $ f=\sum_{i=1}^{r}\lambda_{i}t^{a_{i}} $,  is 
\begin{center}
\textbf{supp}(f)$ =\lbrace a_i \, \in \, \mathbb{Z}^{n}: \, \lambda_{i} \neq 0 \rbrace$.
\end{center}

It is known that given $\mathbb{C}[t,t^{-1}] $ as above, and  let $ \sigma $  be lattice cone. Then
\begin{center}
$ R_{\sigma}=\lbrace f \, \in \, \mathbb{C}[t,t^{-1}] :$ \textbf{supp} $ (f) \subset \, \sigma \rbrace$
\end{center}
  is a  generated finitely  monomial  $ \mathbb{C} -$algebra.\\

\begin{definition}\label{Definition.} An \textbf{affine toric variety} is an irreducible affine variety $ X $ containing  a torus $ T_{N}\simeq (\mathbb{C^{*}})^{n}  $ as Zariski open  subset, such that the action of $ T_{N} $ on itself,  is extended to an algebraic action of $ T_{N} $ on $ X $; that is, there exists   a morphism from $ T_{N} \times X$ to $ X $, \cite{DCoxLittleSchenck2010} .\\

Let $\sigma$  be a lattice cone, the    
   \textbf{affine algebraic scheme:}
\begin{center}
$  X_{\sigma}=$ \textbf{Spec} $ (R_{\sigma}) $.
\end{center}
  is called \textbf{abstract toric affine variety} or \textbf{embedding of torus.}\\
\end{definition}	
For example, set $ 0\leq \, r \, \leq n $, and let   $ \sigma \subset \, \mathbb{R}^{n}$ be a lattice cone generated as follows $ \sigma= Con(e_{1},...,e_{r}) $ where  $ e_{i} $ are  canonical vectors in $ \mathbb{R}^{n} $ for  $ i=1,...,r $. Then computing  its dual cone, one has  $ \sigma^{\vee}= Con(e_{1},...,e_{r}, \pm e_{1},...,\pm e_{n}) $,  and  the   affine toric variety is  
\begin{center}
$  X_{\sigma}=$ \textbf{Spec} $ \mathbb{C}[t_{1},...,t_{r},t_{r+1}^{\pm 1},...,t_{r+n}^{\pm 1}] \simeq \mathbb{C}^{r} \times (\mathbb{C}^{*})^{n-r}$.
\end{center}
this example is seen in, \cite{DCoxLittleSchenck2010}. 
\section{Hilbert Basis.}
The theory of \textbf{Hilbert basis} is an important  algebraic geometry tool. The major contribution of this work 
is  the employment of \textbf{Hilbert basis} associated to a monoid $ N \subset \mathbb{Z}^{n} $ to give explicitly a toric resolution  in  a  set  of  new   coordinates to solve a problem  in singular statistical learning in an original way; also see; \cite{BSturmfels1996}.
\\

Let  $ N \simeq \mathbb{Z}^{n}$ be the lattice  and set $ M= N^{\vee} $ its dual lattice. Let  $ \sigma $ be  a lattice cone defined in $ N $ and let $ \sigma^{\vee} $ be its dual cone  in $ M $. Denote   $ S_{\sigma} = \sigma^{\vee} \cap M$ and note that his monoid is finitely generated  (see Gordan lemma).

  \begin{lemma}\label{Lemma 2.1.1.}  (Basis Hilbert). Set $ \sigma \subseteq \, N $, then $\sigma$   is a   n-dimensional  cone  if and only if it is a strongly convex cone; i.e.,  $ \sigma\cap (-\sigma^{\vee}) = \lbrace 0 \rbrace$. In this case the monoid $ S_{\sigma} $ has a finite minimal set of generators \textbf{H} $ \subseteq \, M \simeq \mathbb{Z}^{d}$ and these are minimal, for details of this proof, see  \cite{DCoxLittleOshea2004},  \cite{DCoxLittleSchenck2010}.\end{lemma}

\begin{definition}\label{Definition 2.1.1} Set $ \omega = (\omega_{1},...,\omega_{n}) \, \in \mathbb{R}^{n} $, and for a polynomial $ f = \sum_{i=1}^n \lambda_{i}t^{a_{i}} $   define its  \textbf{initial form}  $ in_{\omega} (f)$, as the sum over all lambda terms, such that the inner product $ \langle \, \omega, a_{i} \, \rangle $ is maximal. \\
For an ideal $ I $, we mean  the \textbf{initial ideal} as the ideal generated from the initial forms
\begin{center}
$ in_{\omega} (I) =  \langle \, in_{\omega}(f) : f \, \in I \, \rangle$.
\end{center}
\end{definition}
\begin{definition}\label{Definition 2.1.2.} Each  polynomial $ f = \sum_{i =1}^{n}\lambda_{i}t^{a_{i}} $,
in the ring $ \mathbb{C}[t, t^{-1}] $,  is associated to a  convex polytope, or convex hull, in $ \mathbb{R}^{n} $ as follows:
\begin{center}
$  New(f)= $ \textbf{Conv} $ \lbrace a_{i} : i=1,...,m \rbrace \, \subset \mathbb{R}^{n} $.
\end{center}
$ New(f)$ is called the \textbf{Newton polytope} associated to $ f $, in the literature of singular learning machine it is known as the \textbf{exponent space}. But, generally, it is called  the exponent space generated by a Newton polytope; see ref., \cite{BSturmfels1996}.\\
\end{definition}
\begin{lemma}\label{Lemma 2.1.2.} Given $f,g$ two polynomials, then $ New(f*g)= New(f) + New(g) $, where  $*$  is the usual product of polynomials and the sum is the \textbf{Minkowski sum} defined for polytopes, see ref., \cite{BSturmfels1996}. 
\end{lemma}

\begin{proposition}\label{Proposition 2.0.0.} Let  $ I $ be an ideal of the affine toric variety $ X_{\sigma} \, \subseteq \, \mathbb{C}^{n}$. Then define
\begin{center}
$  I(X_{\sigma})= \langle \, t^{l_{+}} - t^{l_{-}} \vert l \, \in L \, \rangle =\langle \, t^{\alpha} - t^{\beta} \vert \alpha - \beta \, \in L, \, \alpha, \, \beta \in \, \mathbb{Z}_{+}^{n} \, \rangle $,
\end{center}
where $ L $ is the kernel of the following morphism $ 0 \longrightarrow L \longrightarrow \mathbb{Z}^{n} \longrightarrow M $ and $ M $ is a monoid  such that $ M \simeq \mathbb{Z}^{d} $. The elements of $ l \, \in L $ satisfies  $ \sum_{i=1}^{n} l_{i}m_{i}=0 $.\end{proposition}

\begin{definition}\label{Definition 2.1.3.} Let be $ L \, \subseteq \, \mathbb{Z}^{n} $, a sub-lattice.\\
\textbf{(a).} The ideal $ I_{L} = \langle \,  t^{\alpha} - t^{\beta} \vert \alpha - \beta \, \in L, \, \alpha, \, \beta \in \, \mathbb{Z}_{+}^{n}  \, \rangle$, is called a \textbf{lattice ideal}.\\
\textbf{(b).} A prime lattice ideal is called a \textbf{toric ideal.}\\\end{definition}
\begin{proposition}\label{Proposition 2.1.0.} An ideal $ I \, \subseteq \, \mathbb{C}[t_{1},...,t_{n}] $ is toric if and only if it is prime and it is generated by binomials.
\end{proposition} One sees the details of the proof, ref. \cite{DCoxLittleSchenck2010}.

\subsection{Toric Morphisms and Gluing Maps.}

\begin{definition}\label{Definition 2.3.0.} Let  $ \Phi :\mathbb{C}^{k} \longrightarrow \Phi(\mathbb{C}^{k}) $ be a monomial map, i.e., each component non zero of $ \Phi $ is a monomial with coordinates in $ \mathbb{C}^{k} $, and let $ X_{\sigma} \hookrightarrow  \mathbb{C}^{k}$ and $ X_{\sigma'} \hookrightarrow  \mathbb{C}^{m}$  be inclusions of toric affine varieties. If $ \Phi (X_{\sigma}) \, \subset \, X_{\sigma'} $, then $ \varphi  := \Phi \vert _{X_{\sigma}}$ is called a \textbf{toric affine morphism} of $ X_{\sigma} $  to $ X_{\sigma'} $. If $ \varphi $ is  bijective and its  inverse map $ \varphi^{-1} :X_{\sigma'} \longrightarrow X_{\sigma} $ is also a toric morphism, then  $ \varphi $ is called an \textbf{affine toric isomorphism} and it is denoted by $ X_{\sigma} \simeq X_{\sigma'} $, \cite{GEwald1993}.
\end{definition}

\begin{proposition}\label{Proposition 2.3.0.} Every toric morphism $ \varphi : X_{\sigma} \longrightarrow X_{\sigma'}$ determines a monomial homomorphism $ \varphi^{*} : R_{\sigma'} \longrightarrow R_{\sigma}  $ and reciprocally, \cite{GEwald1993}.\end{proposition}

\begin{definition}\label{Definition 2.3.1.} For two lattice cones $ \sigma \, \subset \mathbb{R}^{n}= lin(\sigma) $ and $ \sigma' \, \subset  \mathbb{R}^{m}= lin(\sigma') $, we say that $ \sigma $ and $ \sigma' $ are isomorphic  and denote  $ \sigma \, \simeq \, \sigma'$, if $ m= n $ and  there exists an uni modular transformation $ L : \mathbb{R}^{n} \longrightarrow \mathbb{R}^{n} $ such that $ L(\sigma')=\sigma $. The monoids $ \sigma \cap \mathbb{Z}^{n} $ and $ \sigma' \cap \mathbb{Z}^{n} $ are isomorphic also.\end{definition}
\begin{definition}\label{Definition 2.3.2.} Given $ R_{\sigma} $ and $ R_{\sigma'} $ two $ \mathbb{C}-$ algebras, there is a \textbf{monomial isomorphic},    $ R_{\sigma} \simeq R_{\sigma'} $, if there exists an invertible monomial homomorphism $ R_{\sigma} \longleftrightarrow R_{\sigma'} $\end{definition}

\begin{theorem}\label{Theorem 2.3.0.} Set $ \sigma \, \subset \mathbb{R}^{n}= lin(\sigma) $ and $ \sigma' \, \subset \mathbb{R}^{m}= lin(\sigma') $, then  the following conditions are equivalent, \cite{GEwald1993}:
\begin{center}
\textbf{(a).} $ \sigma \, \simeq \, \sigma' $  {} {} {}  \textbf{(b).} $ R_{\sigma} \simeq R_{\sigma'} $  {} {} {} \textbf{(c).} $ X_{\sigma} \simeq X_{\sigma'} $.
\end{center}
\end{theorem}
\begin{proof} The implications a) $ \Rightarrow $ b)$ \Rightarrow $ c) are proven by means of the following diagram and we prove that it is commutative. 
\begin{center}

$ \qquad \qquad \sigma \quad \longrightarrow \quad R_{\sigma}  \quad \hookrightarrow\quad X_{\sigma}=$\textbf{Spec}$(R_{\sigma})$\\
$ \qquad \qquad \downarrow \uparrow  L^{-1}  \qquad  $ $ \downarrow \uparrow   \psi^{-1} \qquad \quad \quad\quad\quad $ $\downarrow \uparrow \varphi^{-1} $\\
$ \qquad \qquad \sigma^{'} \quad \longrightarrow \quad R_{\sigma^{'}} \quad \hookrightarrow\quad X_{\sigma^{'}}=$\textbf{Spec}$(R_{\sigma^{'}})$
\end{center}
We define the monomial homomorphisms $ h_{\sigma}:\sigma \longrightarrow R_{\sigma}$, $ h_{\sigma^{'}}:\sigma^{'} \longrightarrow R_{\sigma^{'}}  $, $ j_{\sigma}:R_{\sigma} \longrightarrow X_{\sigma} $, $ j_{\sigma^{'}}:R_{\sigma'} \longrightarrow X_{\sigma'} $. From the hypothesis  one has that $ \sigma \backsimeq \sigma^{'} $ then there  exists a uni-modular transformation $ L $ such that $ L(\sigma)=\sigma^{'} $ and its  inverse transformation $ L^{-1}(\sigma^{'})=\sigma $ is well defined. Then one does the following monomial homomorphisms:

 \begin{center}
 $  h_{\sigma}(a)= \sum \lambda_{a}t^{a} \, \in \, R_{\sigma}$ and $ a\in \, supp(h_{\sigma}) \subset \sigma $,\\
 $ h_{\sigma^{'}}(a')= \sum \lambda_{a'}t^{a'} \, \in \, R_{\sigma^{'}} $ and $a' \in \, supp(h_{\sigma^{'}}) \subset \sigma^{'}   $,\\
 $ \psi (h_{\sigma}(a))=\sum \lambda_{a'}t^{L(a)} \, \in \, R_{\sigma^{'}}$ and $ L(a)=a' \in \, supp(\psi) \subset \sigma^{'}  $,\\
 $ \psi^{-1} (h_{\sigma^{'}}(a'))=\sum \lambda_{a}t^{L^{-1}(a')} \, \in \, R_{\sigma}$ and $ L^{-1}(a')=a \in \, supp(\psi^{-1}) \subset \sigma  $.\\
 \end{center}
Choosing the prime generators $ t^{a} \, \in \, R_{\sigma}, \, a \,\in\, \sigma$ and $ t^{a'} \, \in \, R_{\sigma^{'}}, \, a' \,\in\, \sigma^{'} $, define:
\begin{center}
$  j_{\sigma}(t^{a})=\langle t^{a} \rangle \, \in \, X_{\sigma}=$\textbf{Spec}$(R_{\sigma})$,\\
$  j_{\sigma^{'}}(t^{a'})=\langle t^{a'} \rangle \, \in \, X_{\sigma^{'}}=$\textbf{Spec}$(R_{\sigma^{'}})$,\\
$  \varphi(\langle t^{a} \rangle)=\langle t^{L(a)} \rangle \, \in \, X_{\sigma^{'}}=$\textbf{Spec}$(R_{\sigma^{'}})$,\\
$  \varphi^{-1}(\langle t^{'a} \rangle)=\langle t^{L^{-1}(a')} \rangle \, \in \, X_{\sigma}=$\textbf{Spec}$(R_{\sigma})$.
\end{center}
where  $ \langle t^{L(a)} \rangle $ and  $ \langle t^{L^{-1}(a')} \rangle $  are prime ideals like a realization, respectively, of the spectrum \textbf{Spec} of the coordinate rings $ R_{\sigma} $ and $ R_{\sigma^{'}} $.
 
We see easily that these monomial homomorphisms accomplish the following identities, without lost of generality,  $ \lambda_{a}=\lambda_{a'}=1 $ so: $ L\circ L^{-1}=id_{\sigma} $, $ L^{-1}\circ L=id_{\sigma^{'}} $, $ \psi \circ \psi^{-1}=id_{R_{\sigma^{'}}} $, $ \psi^{-1} \circ \psi=id_{R_{\sigma}} $, $ \varphi^{-1}\circ \varphi = id_{X_{\sigma}} $, $\varphi \circ \varphi^{-1} = id_{X_{\sigma^{'}}}  $. 
The isomorphisms $ L $, $ \psi $, $ \varphi $, are isomorphisms of, cones, algebras of coordinate rings, and isomorphisms of toric varieties (toric morphism) respectively, and the first are well defined; one is a  uni-modular transformation and the second one is an isomorphisms of algebras. It only remains to proof $ \varphi $ is a toric morphism. Define the monomial homomorphism $ \Phi:\mathbb{C}^{n} \longrightarrow \mathbb{C}^{n} $ by $ \Phi(\langle t^{a} \rangle)=\langle t^{a'} \rangle\, \ni \, \Phi(X_{\sigma}) \subset X_{\sigma^{'}} $. This homomorphism induces the morphism $ \varphi $ which is bijective, so that for the generator $ t^{0}=1_{X_{\sigma}} $ as lattice vector $ a=0\, \in \, \sigma $, $ \varphi(t^{0})=t^{L(0)}=1_{X_{\sigma^{'}}} $. Then $ \varphi $ is injective and consider  the generator $ t^{a'}\, \in \, X_{\sigma^{'}} $. Since  $ L(a)=a'\Rightarrow \,\exists \,t^{a} \, \in \, X_{\sigma} \, \ni \,\varphi(t^{a})=t^{L(a)}=t^{a'} $, then $ \varphi $ is surjective. Note that   $ \varphi=\Phi\mid _{X_{\sigma}} $; in the same way one can see that $ \varphi^{-1} $ is a toric morphism too. Therefore, $ \varphi $ is a toric isomorphism. On the other hand,  $ \psi(h_{\sigma}(a))=h_{\sigma^{'}}(L(a)) $, $ \varphi(j_{\sigma}(t^{a}))=j_{\sigma^{'}}(t^{a'}) $; which proves that the   diagram commutes and one  obtains the isomorphisms wished. For  details of the  implication c) $ \Rightarrow $ a), see \cite{GEwald1993}. \textbf{q.e.d.} \end{proof}

 
\begin{definition}\label{Definition 2.3.3.} Recall  that a  \textbf{complex projective n-space $ \mathbb{C}P^{n} $} is the space of class of equivalence of pairs of points such that it consists of lines on $ \mathbb{C}P^{n}=\mathbb{C}^{n+1} / \sim  $. The relationship between points $ \sim $ is of the following manner, given any vector $ v:= (\eta_{0},...,\eta_{n}) $ it defines a line $ \mathbb{C}*v $ and two of said vectors $ v \sim v' \in \mathbb{C}^{n+1}\setminus \lbrace 0 \rbrace $ define the same line if and only if, one is a scalar multiple of the other.\end{definition}\
In the next example we point the important relationship of basis Hilbert and the Theorem \ref{Theorem 2.3.0.}, this connection is the great importance for the applications in the following sections; the example of Hirzebruch surface is possible consulting it in; Ewald, \cite{GEwald1993}.
\begin{example}\label{Example}. By $ H_{k} $ we mean the \textbf{Hirzebruch surface}. We consider a hyper surface in $ \mathbb{C}P^{1}\times \mathbb{C}P^{2}=\lbrace ([\eta_{0},\eta_{1}],[\zeta_{0},\zeta_{1},\zeta_{2}]) : (\eta_{0},\eta_{1}) \neq (0,0),(\zeta_{0},\zeta_{1},\zeta_{2})\neq (0,0,0) \rbrace$ determined by the equation, see example given in \cite{GEwald1993},\\
\begin{center}
$ \eta^{k}_{0}\zeta_{0}=\eta^{k}_{1}\zeta_{1},\quad k \in \mathbb{Z}. $
\end{center}
Applying Theorem \ref{Theorem 2.3.0.}, as in previous examples,  one  has the  isomorphic coordinate rings  by each one of the affine charts associated to this surface. Determining  the Newton polytopes of fan $ \Sigma $,so as its dual cones, it follows there four planes which are affine charts, and its gluing depend of $ k $, thus,\\
\begin{center}
$ R_{\sigma^{\vee}_{0}}= \mathbb{C}[z^{e_{1}},z^{e_{2}}] = \mathbb{C}[z_{1},z_{2}]$,\\
$ R_{\sigma^{\vee}_{1}}= \mathbb{C}[z^{-e_{1}},z^{e_{1}+ke_{2}}]=\mathbb{C}[z^{-1}_{2},z_{1}z^{k}_{2}]$,\\
$  R_{\sigma^{\vee}_{2}} = \mathbb{C}[z^{-e_{1}},z^{e_{2}}]=\mathbb{C}[z^{-1}_{1},z_{2}]$,\\
$ R_{\sigma^{\vee}_{3}} =\mathbb{C}[z^{-e_{1}-ke_{2}},z^{-e_{2}}]=\mathbb{C}[z^{-1}_{1}z^{-k}_{2},z^{-1}_{2}] $;\\

which implies the following toric varieties:\\

$X_{\sigma^{\vee}_{0}}=$\textbf{Spec}$ (\mathbb{C}[z_{1},z_{2}])$ ;\\ 
$X_{\sigma^{\vee}_{1}}=$\textbf{Spec}$ (\mathbb{C}[z^{-1}_{2},z_{1}z^{k}_{2}])$ ;\\
 $X_{\sigma^{\vee}_{2}}=$\textbf{Spec}$ (\mathbb{C}[z^{-1}_{1},z_{2}]) $;\\
 $X_{\sigma^{\vee}_{3}}=$\textbf{Spec}$ (\mathbb{C}[z^{-1}_{1}z^{-k}_{2},z^{-1}_{2}]) $.
\end{center}
Where the fan $ \Sigma^{\vee} $ formed by the cones $\sigma^{\vee}_{0}=Con(e_{1},e_{2})  $, $ \sigma^{\vee}_{1}=Con(-e_{1},e_{1}+ke_{2}) $, $ \sigma^{\vee}_{2}= Con(-e_{1},e_{2}) $, $ \sigma^{\vee}_{3}=Con(-e_{1}-ke_{2},-e_{2} )$ are the Hilbert basis $ H_{\Sigma} $ associated to the fan $ \Sigma=\lbrace \sigma_{0}, \sigma_{1}, \sigma_{2} , \sigma_{3} \rbrace  $, with $ \sigma_{0}= $Con$ (e_{1}, e_{2}) $, $ \sigma_{1}= $Con$ (-e_{2}, -ke_{1} + e_{2}) $, $ \sigma_{2}= $Con$ (-e_{2}, e_{1}) $, $ \sigma_{3} =$Con$ (-ke_{1} + e_{2}, e_{1})$. This technique will be applied in the example of three layer perceptron.\\

Following this re parametrization of polynomials one can see that it is convenient to work in complex projective spaces according to Theorem 1.\end{example}

\textbf{Gluing Maps.}\\
\begin{lemma}\label{Lemma 2.3.0.} Let $ \sigma $  be a lattice cone and set $ \tau \, \preceq \, \sigma  $. The natural identification,
\begin{center}
$  X_{\tau^{\vee}} \: \simeq \: X_{\sigma^{\vee}} \setminus \lbrace  u_{k}= 0 \rbrace$.
\end{center}
where $ u_{k} $ is the last generator of the representation of the coordinate ring associated to $ X_{\sigma^{\vee}} $, see details of the proof, ref., \cite{GEwald1993}.\end{lemma}
\begin{definition} \label{Definition 2.3.3.1.} The isomorphism,
\begin{center}
$ \psi_{\sigma, \sigma'} : \,  X_{\sigma^{\vee}} \setminus \lbrace  u_{k}= 0 \rbrace \longrightarrow  X_{\sigma'^{\vee}} \setminus \lbrace  v_{l}= 0 \rbrace $.
\end{center}
is called \textbf{gluing morphism}, which glues the varieties $ X_{\sigma^{\vee}} $ and $ X_{\sigma'^{\vee}} $ in the variety $ X_{\tau^{\vee}} $.\end{definition}
\subsection{Toric Resolution.}
\begin{definition}\label{Definition 2.4.0.} (Singularity) Let $ X_{\sum} $  be a n-dimensional toric variety and let $ \sum $ be a re\-gu\-lar fan. A point $ p \, \in \, X_{\sum} $ is called \textbf{singular} or \textbf{singularity} of $ X_{\sum} $, if $p$ belongs to an  affine chart $ X_{\sigma^{\vee}} $  where $ \sigma \, \in \, \sum $ which is not of the form $ \mathbb{C}^{k} \times (\mathbb{C}^{*})^{n-k} $. For details of the proof, see ref., \cite{GEwald1993}.\\
\end{definition} 

\begin{theorem}\label{Theorem of Hironaka-Atiyah} (Hironaka-Atiyah) Let $ f $ be a real analytical function in a neighborhood  of $ \omega = (\omega_{1},....,\omega_{n})\in \mathbb{R}^{n} $ such that  $ f(\omega)= 0 $. Then there exists an open set $ V \subset\mathbb R$, a real analytical variety $ U $ and a proper analytical map $ g :U  \to V$ such that:\\
(a) $ g\quad :U-\epsilon \longrightarrow V-f^{-1}(0) $ is an isomorphism, where  $ \epsilon =g^{-1}(f(0)) $,\\
(b) For each $ u \in U $, there  exist local analytical coordinates  $ (u_{1},....,u_{n}) $ such that $ f(g(u))=\pm u^{s_{1}}_{1}u^{s_{2}}_{2}*...*u^{s_{n}}_{n} $, where $ s_{1},....,s_{n} $ are  non negative integers; see ref. \cite{SWatanabe52001}.
\end{theorem}

The previous theorem is a version of the well-known theorem of resolution of singula\-ri\-ties established by  Hironaka in algebraic geometry, see, ref. \cite{HHironaka1964}, \cite{SWatanabe22001}.\\

\begin{theorem}\label{Theorem of toric modification} Let $ X_{\sum} $ be a regular toric variety, and let $ X_{\sum_{0}} $ be a toric invariant sub variety defined by the star $ st(\sigma,\sum) \backsimeq  \sum_{0}  $ of $ \sigma $ into $ \sum $; $ 1 < k:=dim \sigma \leqslant n $.\\
(a) Under toric blow up $ \psi^{-1}_{\sigma} $, any point $ x \in X_{\sum_{0}} $ is  substituted by a k-dimensional \textbf{(k-1) projective space.}\\
(b) The blow down $ \psi_{\sigma} $ is a toric morphism which is bijective in the outside of $ \psi^{-1}_{\sigma} $.\\
See ref. for the proof of this fact, \cite{GEwald1993}.
\end{theorem}

\section{Singular Statistical Learning.}

In this section we will focus on the statistical learning machine. Given a probability space $ (\Omega, \mathbb{F}, P)  $, where $ \Omega $ is a set of events, $ \mathbb{F} $  is a $ \sigma-$algebra on $ \mathbb{F} $, and  $ P $ is a measure of Kolmogorov, one can compute  the predictive probability $ P(y \vert x,\omega) $ of the output variable $ y \in \mathbb{R}^{n}$ given $ x \in \mathbb{R}^{m}$ and a parameter vector $ \omega \, \Theta \subseteq \mathbb{R}^{n}$.  This probability is factorized according to  Bayes theorem and  to Theorem of Hammersley and Clifford; that is,   $ P(X \vert Y, \Theta) $ satisfies the local property of Markov and it can be factorized through an undirected graph $ G =(E,V)$, also it is represented by  a toric variety. \\

\begin{definition}\label{Definition 4.1.7} (Identifiable and non identifiable Machines) Let  $ (\Omega, \mathbb{F}, P)  $ be a probability space and let $ P(y \vert x,\omega) $  be the probabilistic inference or prediction probability of a statistical machine, where $ y \in \mathbb{R}^{n} $ is an     output vector  and  $ x \in \mathbb{R}^{m} $  is an input vector. If $  \omega \mapsto P(y \vert x,\omega) $ is an  injective mapping is we say the machine is \textbf{identifiable}, see  Watanabe \cite{SWatanabe42001}. If the mapping is not injective, then we say that the machine is a \textbf{ non identifiable machine}. \end{definition}

The probability densities of the learning machines are defined in the probability space $ (\Omega, \mathbb{F}, P)  $ and are denoted as follows:  
\begin{itemize}
\item \textbf{prediction of the vector } $P(y \vert x, \omega)  $, {} $ y \in \, \mathbb{R}^{M}$ , 
\item \textbf{true inference of the machine} $ q(y\vert x) $, 
\item  $ q(y \vert x)q(x) $ is the \textbf{ distribution of probability with which are taken and trained the set of examples} of inference machines in an independent way.
\end{itemize} 
 
\begin{definition}\label{Definition 4.1.8} Let $ \omega_{0} \in \Theta \, \subset \,\Omega$   be a parameter such that $ P(y \vert x,\omega_{0})=q(y \vert x) $; which means, that the parameter $ q(y \vert x) $    (which establishes   the true inference of the statistical machine)  is equal  to the  predictive probability density of the output vector $ y \in \mathbb{R}^{n} $. For non identifiable machines this parameter is not unique. Even more, the set of these  parameters is called \textbf{space of true parameters}  and it is denoted  by \\
\begin{center}
$ W_{0}= \lbrace \omega_{0} \in \Theta \, \subset \, \Omega :P(y \vert x,\omega_{0})=q(y \vert x) \rbrace $. 
\end{center}
It is well known that $W_{0}$ is a  sub-variety formed by singular points. If these probability densities are analytic functions, then $W_0$ is called an \textbf{analytic set}. But, if these probability densities are polynomials, then $W_0$ is called an \textbf{algebraic set}. These sets are very important for our study.
\end{definition}

\textbf{Watanabe Theorems.}\\

\begin{theorem}\label{(1) Theorem, Watanabe,} (1)Theorem, Watanabe, \cite{SWatanabe52001}, \cite{SWatanabe2005}. Suppose that $ f $ is an analytic function and $ \varphi$ is a probability density function both defined in $  \mathbb{R}^{d} $. Then, there exists a real constant $ C $ such that
\begin{center}
$ G(n) \leq \lambda_{1} \log n - (m_{1} - 1)\log (\log n) + C$,
\end{center}
 for any natural number $ n $. The rational number $ -\lambda_{1}$ {} $(\lambda_{1} > 0) $ and the natural number $ m_{1} $ are the  largest poles   of a  meromorphic function which is analytical continuation of  
\begin{center}
$ J(\lambda)= \displaystyle \int_{f(\omega)< \epsilon} f(\omega)^{\lambda}\varphi'(\omega) d\omega, \quad (Re(\lambda)> 0), $
\end{center}
where $ \epsilon > 0 $ is a constant, and $ \varphi'(\omega) $ is a function of class $C^{\infty}_{0}  $ satisfying  $ 0 \leqslant \varphi'(\omega) \leqslant \varphi(\omega)$.\end{theorem}

\begin{definition}\label{Definition 4.3.1.} The poles of  the function $ J $ belong to the intersection between the negative real semi axes and the set $ \lbrace m+\nu;  \ m=0,-1,-2,...,b(\nu)=0  \rbrace $. Denoting these poles in a decreasing manner:  $ -\lambda_{1},-\lambda_{2},-\lambda_{3},...,-\lambda_{k} $, where $ \lambda_{k} $ is a  rational number, and the  multiplicity of $ -\lambda_{k} $ is denoted by $ m_{k} $.
\end{definition}

\textbf{ Condition (A)}. Let $ \psi(x,\omega) $  be  a  real valued function, where  $ (x,\omega) \in \mathbb{R}^{M} \times \mathbb{R}^{d} $, such that:\\
(1) $ \psi(x,\cdot) $ is an analytic function  on  $ W=supp(\varphi) \subset \mathbb{R}^{d} $ which can be extended to a holomorphic function on some open set $ W^{*} $, where $ W\subset W^{*} \subset \mathbb{C}^{d} $, and  $ W^{*} $ is independent of $ x \in supp(q) \subset \mathbb{R}^{M}$.\\
(2) $ \psi(\cdot,\omega) $ is a measure function on  $ \mathbb{R}^{M} $, which  satisfies:
\begin{center}
$ \displaystyle \int sup_{\omega \in W^{*}} \Vert \psi(x,\omega) \Vert^{2}q(x) dx <  \infty $,
\end{center}
where $ \Vert \bullet \Vert $ is the norm of the vector $ \psi(x,\omega) $.\\

\begin{theorem}\label{(2) Theorem, Watanabe} (2) Theorem, Watanabe \cite{SWatanabe42001}, \cite{SWatanabe52001}. Set a constant $ \sigma > 0 $. Let $ \varphi $  be a probability density of class $ C^{\infty}_{0} $.  We will consider   the statistical learning machines characterized by the following   true inference of machine:
\begin{center}
$ P(y \vert x,\omega) =\dfrac{1}{(2\pi \sigma^{2})^{N/2}} \exp \left( \dfrac{-\Vert y - \psi(x,\omega)\Vert^{2}}{2\sigma^{2}}\right) $,
\end{center}
where both $ \psi(x,\omega) $ y $ \Vert \psi(x,\omega) \Vert^{2} $ satisfies the condition (A). Then there exists a constant $ C' > 0 $ such that \begin{center}
$ \vert G(n) - \lambda_{1} \log n + (m_{1} - 1)\log \log n \vert \leqslant C'$,
\end{center}
for any natural number $ n $,  where the rational number $ -\lambda_{1} (\lambda_{1} > 0) $ and a natural number $ m_{1} $ are the  largest poles   of a  meromorphic function which is analytical continuation of 
\begin{center}
$\displaystyle  J(\lambda)= \int _{f(\omega) < \epsilon} f(\omega)^{\lambda} \varphi(\omega) d\omega, (Re (\lambda)> 0) $,
\end{center}
where $ \epsilon > 0 $ is a constant.\end{theorem}

\textbf{Learning Curves and Resolution of Singularities.}\\
It is well known that  the  regular statistical models in which  $ \lambda_{1}=d/2 $ and $ m_{1}=1 $ are special cases of Theorem (2) (Watanabe). In models of non identifiable machines, generally, the  bayesian neural networks have different values of $ \lambda_{1} \leq d $  and  $ m_{1} \geq 1 $.  It is work of the algebraic geometry to find the poles $ \lambda_{1} $ and $ m_{1} $, of the meromorphic function $ J(\lambda) $ defined in theorem (1) and (2) (Watanabe),  by means of   techniques of resolution of singularities suggested by Watanabe, \cite{SWatanabe42001}, \cite{SWatanabe52001}, \cite{SWatanabe62001}, such as toric modification and blow up, in the algebraic set $ \lbrace \lambda \in W : H(\lambda)=J(\lambda) =0 \rbrace. $ .\\

\begin{corollary}\label{Corollary 4.3.1} Suppose the hypothesis of Theorem  \ref{(2) Theorem, Watanabe}. If $ c(n+1) - c(n)= o\left( \dfrac{1}{n\log n}\right)  $, then the \textbf{learning curve} is given by, 
\begin{center}
$ K(n)= \dfrac{\lambda_{1}}{n} + \dfrac{m_{1} - 1}{n \log n} + o\left( \dfrac{1}{n\log n}\right)  $.
\end{center}
Using this formula in regular models one has that  $ \lambda_{1} =d/2$ and $ m_{1}=1 $. For non identifiable models, such as bayesian neural networks, the corresponding values are  $ \lambda_{1} \leq d/2 $ y $ m_{1} \geq 1 $.\\
\end{corollary}

\begin{corollary}
\label{Corollary 4.3.2.} Suppose the hypothesis of Theorem \ref{(1) Theorem, Watanabe,}. If  $ \varphi'(\omega) > 0 $ for each $ \omega_{0} \in W_{0} $,  then $ \lambda_{1} \leq d/2 $ where $ d $ is the \textbf{dimension of parameter spaces}. See ref. \cite{SWatanabe42001}, \cite{SWatanabe52001}.\end{corollary}

\begin{definition}\label{Definition 4.1.9} The \textbf{Kullback distance}, or  \textbf{information entropy}, of a statistical machine  
quantifies  the distance between the predictive probability $ P(y \vert x,\omega) $, of the output variable $ y \in \mathbb{R}^{N}$, and the true statistical inference of the machine $q(y \vert x)  $. \\
\begin{center}
\textbf{(Kullback distance)} $\displaystyle H(\omega) = \int \log \dfrac{q(y \vert x)}{P(y \vert x,\omega)}q(y \vert x)q(x) dxdy$.
\end{center}
where $ q(x) $ is the true probability of the input variable $ x $.\end{definition}

The Kullback distance induces other important definitions. 

\begin{definition}\label{Definition 4.1.10.}  The \textbf{learning curve} of a statistical machine or \textbf{generalization of the error}, Watanabe \cite{SWatanabe32001}, is given by\\
\begin{center}
$ \displaystyle K(n)= E_{n} \left\lbrace  \int \log \dfrac{q(y \vert x)}{P_{n}(y \vert x,\omega)}q(y \vert x)q(x) dxdy \right\rbrace  $,
\end{center}
where $ E_{n} \lbrace \bullet \rbrace $ is the expected value over all pairs of trained examples  by the machine, and $ P_{n}(y \vert x,\omega) $ is the mean density probability over all posterior probabilities of the output of the machine. 
\end{definition}

\textbf{Algebraic geometry of Statistical machines.}\\

\begin{definition}\label{Definition 4.1.11} It is important to comment that an \textbf{algebraic set} $ W_{0} $ is, equivalently, defined by 
\begin{center}
$ W_{0}= \lbrace H(\omega)=0 : \omega \in \Theta \rbrace$.
\end{center}
This set is not empty and is the principal set of our  study related with   singular machines.
\end{definition}

\section{Singular Machines.}

\begin{definition}\label{Definition 4.2.1} Let  $(\Omega, \mathbb{F}, P) $ be  a probability space and let  $ y \in \mathbb{R}^{N}$  be a random vector. Define the \textbf{Fisher information matrix}, as follows, see A.S. Poznyak, \cite{ASPoznyak2009}: 

\begin{center}
$\displaystyle  I(\omega) = E_{n} \left\lbrace \nabla_{\omega} \log P_{n}(y \vert x,\omega)   \nabla_{\omega}^{\intercal} \log  P_{n}(y \vert x,\omega) \right\rbrace = \, \, \int_{\omega \in \mathbb{R}^{M}} \left\lbrace \nabla_{\omega} \log P_{n}(y \vert x,\omega)   \nabla_{\omega}^{\intercal} \log  P_{n}(y \vert x,\omega) \right\rbrace P_{n}(y \vert x,\omega)q(x) dxdy $,
\end{center}
where $ P_{n} $ is given  in Definition \ref{Definition 4.1.10.} In general terms, this expression can be understood as a metric in the parameter space whenever the matrix is  positive defined.
\end{definition}
 
\begin{definition}\label{Definition 4.2.2.} A statistical learning machine is called \textbf{regular learning machine} if the Fisher information matrix is positive defined, otherwise, it is called  a \textbf{singular learning machine} if there  exists a parameter $ \omega \in \Theta $ (called singularity of the Fisher information matrix) such that \textbf{det}$ I(\omega)=0 $. These singularities are several  and the probability of the parameter $ \omega $ can not be approximated by a quadratic form in the sense of differential geometry, see the regular statistical machines, see ref. \cite{SWatanabe42001}.\end{definition}

\subsection{Effect of the Singularities in the Statistical Learning.}

In the following we define the mean empirical Kullback distance as:
\begin{center}
$ H_{n} = \dfrac{1}{n} \sum_{i=1}^{n} \log \dfrac{q(y_{i} \vert x_{i})}{P(y_{i} \vert x_{i},\omega)}$,
\end{center}
and let $ H(\omega) $  be an usual Kullback distance as we previously saw, if there exists a parameter $ \omega_{0} $ such that $ H(\omega_{0})=0 $, then $ H(\omega) $ satisfies the statement of Theorem \ref{Theorem of Hironaka-Atiyah}. Therefore,  there exists a variety $ U $ and  a resolution map $ g: U \mapsto W $, such that,
\begin{center}
$ H(g(u))=A(u)^{2} $ with $ A(u)= u^{k_{1}}_{1}*...*u^{k_{d}}_{d} $,
\end{center}
and the empirical distance can be written as above; for more details see ref. \cite{SWatanabe32001}, \cite{SWatanabe42001}, \cite{SWatanabe62001}.

\begin{definition}\label{Definition (Synaptic Function)} According to the previous notations, the synaptic function of a statistical learning machine is given as follows:
 \begin{center}
$ \psi(x,y,u)=\dfrac{1}{A(u)}\left( H(g(u)) - \log\dfrac{q(y \vert x)}{P(y \vert x,g(u))}\right)  $.
\end{center}
 The function $ \psi(x,y,u) $ can be written as $ \psi(x,y,g^{-1}(u)) $ if $ H(\omega)\neq 0 $. However, it is well defined, in general, when  $ H(\omega)=0 $. On the other hand, we proved that $ \psi(x,y,u) $ is an analytic function of $ u $, whenever $ H(g(u))=0 $. From the property of normal crosses of $ A(u) $, one can see that   $ \psi(x,y,u) $ is well defined in the variable $ u $, see ref., \cite{SWatanabe42001}, \cite{SWatanabe52001}.\end{definition}
 \textbf{Learning Coefficient.} In this part of our work, we compute the learning coefficient of the following statistical learning machine:\\
\begin{center}
$P(y \vert x,a,b) = \dfrac{1}{\sqrt{2\pi}} \exp(-\dfrac{1}{2}(y - af(b,x))^{2})$.
\end{center}
The true statistical inference of the machine is given by:
\begin{center}
$ q(y \vert x) =\dfrac{1}{\sqrt{2\pi}} \exp\left( \dfrac{-1}{2}(y -\dfrac{a_{0}f(b_{0},x)}{\sqrt{n}})^{2}\right) $,
\end{center} 
where
\begin{center}
$ \int \dfrac{\psi(b)db}{\Vert f(b) \Vert} < \infty $.
\end{center}
Then, the learning curve of this machine can be expanded asymptotically by: 
\begin{center}
$ G(n)=\dfrac{\lambda(a_{0},b_{0})}{n} + o\left( \dfrac{1}{n}\right) $.
\end{center}
The \textbf{learning coefficient}, which is independent of $ n $,  is given  by: 
\begin{center}
$ \lambda(a_{0},b_{0})=\dfrac{1}{2} \left\lbrace  1 + a^{2}_{0}\Vert f(b_{0})\Vert^{2} - \sum_{j=1}^{J}a_{0}f_{j}(b_{0})E_{g}\left[ \dfrac{\partial}{\partial g_{j}} \log Z(g)\right]  \right\rbrace $,
\end{center}
 see, Watanabe, \cite{SWatanabe42001}, where $ g=\lbrace g_{j} \rbrace$ is a random variable  subject to the dimensional gaussian distribution $ J $,  whose mean is zero and its covariance matrix is the identity. By  $ E_{g} $ we mean  the expecting values on $ g $, and set
$(g) = \int \exp(L(g))\dfrac{\psi(b)db}{\Vert f(b) \Vert},$
with $  L(b)=\dfrac{m((g + a_{0}f(b_{0}))*f(b))^{2}}{2\Vert f(b) \Vert^{2}} .$

Now, considering,  at the beginning of our example, the synaptic function $ f(b,x) $ in terms of its expansion in orthonormal basis $ e_{j} $, one has that
\begin{center}
$ P(y \vert x,a,b)=\dfrac{1}{\sqrt{2\pi}} \exp \left( \dfrac{-1}{2}(y - \sum_{j=1}^{N} ab_{j}e_{j}(x))^{2}\right).  $
\end{center}
In this case for the synaptic function $ \Psi $. When $ N \geqslant 2 $, in a model of true regression, see ref. \cite{SWatanabe42001}, this machine is singular with parameter space $ W_{0}=\lbrace (a,b), a=0, b=0, : a \in \mathbb{R}, b \in \mathbb{R}^{N} \rbrace$. Using the toric resolution, it is  parametrized by means $ \omega_{i}=ab_{i} $ and substituting in the model one has:
\begin{center}
$ P(y \vert x,a,b)=\dfrac{1}{\sqrt{2\pi}} \exp\left( \dfrac{-1}{2}(y - \sum_{j=1}^{N} \omega_{j}e_{j}(x))^{2}\right)  $.
\end{center}
With this resolution in the parameter space,  the model has become a regular model, with learning coefficient given by $ \lambda(\omega_{0}) = N/2$ for an arbitrary parameter $ \omega_{0} $. Meanwhile, without this toric resolution the learning coefficient would be given by the above expression  $ \lambda( a_{0} ,b_{0}) $. Clearly, $ \lambda(a_{0},b_{0}) \neq \lambda(\omega_{0})  $. With the previous facts, one see that  the \textbf{singularities into the parameter space play an important role in the statistical learning, \cite{SWatanabe22001}, \cite{SWatanabe32001}.}

\begin{theorem}\label{Corollary 4.3.3} Let  $P(y\vert x, \omega)  $ be   a non singular  statistical learning machine, see Definition \ref{Definition 4.2.2.}, in the probability space $ (\Omega, \mathbb{F}, P)  $, and  consider the  Kullback distance associated with its parameters space; i.e.,   $ \lbrace \lambda \in W \, \subset \, \Omega: H(\lambda)=J(\lambda) =0 \rbrace $. Then,  the following polynomial is a parametrization such that for  each  $ \Theta \, \subset \, \Omega$, there exists $ \, \omega \,\in \, \Theta \subset \Omega \subset \mathbb{C}^{n},\,   $ as we have seen, $ H(\omega)=\sum_{i=1}^{n} c_{i}\omega^{a} $  with $ c_{i} \,\in \, \mathbb{R} ,\, \omega^{a} \, \in \, \mathbb{C^{*}}$, {} for each $i$, and  $ a \, \in \, \mathbb{Z}^{n} $ being the lattice vector (see Definition \ref{Definition 2.1.2.}, Newton polytope) if and only if the lattice cone $ \sigma \subseteq \, supp(H(\omega)) $, generated by the lattice vector of support  of $ H(\omega) $, is not singular; i.e.,
\textbf{Det} $ (\sigma)=1 $.\end{theorem}
\begin{proof}  The need is a consequence of the following facts. As  $ H(\omega) $ is a re parametrization of singular polynomial $ H(\lambda) $ then  there exists  a resolution map $ g: H(\lambda) \, \longrightarrow \, H(\omega) $, by the Theorem \ref{Theorem of Hironaka-Atiyah}, such that $ H(\omega) $ is not singular. Now, let $ \sigma' \subset \, supp(H(\lambda)) $ and let  $ \sigma \subset \, supp(H(\omega)) $ be  lattice cones generated by the lattice vectors of support of $ H(\lambda) $ and $ H(\omega) $, respectively. Then, we affirm that $ g $ induces  a morphism of monomial generated finitely $ \mathbb{C}-$ algebras: $ R_{\sigma'}= \lbrace f \vert \, supp(f)\subset \, \sigma' \rbrace $ and $ R_{\sigma}= \lbrace f \vert \, supp(f)\subset \, \sigma \rbrace  $,  such that $ H(\lambda)\,\in \, R_{\sigma'} $ and $ H(\omega) \, \in \, R_{\sigma}$. The proof of this  follows from its definition as resolution map. Then Theorem \ref{Theorem 2.3.0.}, implies that  $ R_{\sigma'} \simeq R_{\sigma} \, \Longrightarrow \, \sigma' \simeq \sigma$. Therefore there exists  an uni modular transformation $ L \, \in \, \mathbb{Z}^{n} \times \mathbb{Z}^{n} $ with a matrix associated to the Hilbert basis $ H_{(\sigma')^{\vee}} $, with dual cone $ (\sigma')^{\vee} \, \Longrightarrow \, L(\sigma')=\sigma$ and with uni modularity of this transformation. Then \textbf{Det} $L(\sigma')=$ \textbf{Det} $ \sigma=1 \, \Longrightarrow \,\sigma$ is not singular.\\
 Reciprocally, let $ \sigma \subset \, supp(H(\omega)) \, \ni $ \textbf{Det} $ \sigma =1 \, \Longrightarrow \,\exists \, L \in \, \mathbb{Z}^{n} \times \mathbb{Z}^{n} \, \ni \, L(\sigma')=\sigma $ be a lattice cone  with $ L $ being an uni modular transformation, for some cone $ \sigma' \subset \, \mathbb{Z}^{n} \, \Longrightarrow \, \sigma'\simeq \sigma$. Then by Theorem \ref{Theorem 2.3.0.}, one has that this isomorphism lifts to a toric morphism $ \psi $ such that $ R_{\sigma'}\simeq R_{\sigma} $ to make it compatible with the previous notation, we do $ g=\psi $. Since the morphisms of monomial finitely generated $ \mathbb{C}-$ algebras are torics as it has been said before, then $ g $ is a resolution map of $ H(\omega) $ and therefore  this polynomial is not singular. We will prove that the Fisher information matrix $ I(\omega) $ associated to the model is not singular, i.e., \textbf{Det} $ I(\omega)\neq 0 $.\\
\textbf{Proof of the affirmation.} If $ H(\omega) $ is singular, then  the inference machine  $ P(y \vert x,\omega) $ is a model non iden\-ti\-fia\-ble, by Definition \ref{Definition 4.1.8}. Therefore, $ P(y \vert x,\omega)=q(y\vert x) $ where $ q(y \vert x) $ is the true statistical inference of the model. Applying the operator $ \nabla_{\omega} $ to  $P(y \vert x,\omega)  $, from Definition of Fisher information matrix (Definition \ref{Definition 4.2.1}), one has that  $\nabla_{\omega} P(y \vert x,\omega)= \nabla_{\omega} q(y\vert x) = 0  \,\Longrightarrow$ \textbf{Det} 
$ I(\omega)=0 $, which concludes the proof, \textbf{q.e.d.} \end{proof}

\section{Applications in singular machines.}

We present an application to the learning curve in the following singular machine.\\
\textbf{Application A.} Consider  the polynomial which represents the learning curve of a perceptron of two layers $ H(a,b,c)= a^{2}b^{2}+ 2abc + c^{2} +  3a^{2}b^{4}, \ (a,b,c) \in \mathbb{R}^{3} $ which is singular  in its parameter space, $ (0,0,0) \in \mathbb{R}^{3} $. Establishing the Newton polytope,   defined by $ supp(H) $, one has the lattice cone: $ \sigma=Con ((2,2,0),(1,1,1),(0,0,2),(2,4,0)) $.\\
Now,  we get the following associated dual cone, see Theorem \ref{Theorem 2.3.0.}:
\begin{center}
$ \sigma^{\vee}= Con(2e_{1} - e_{2} - e_{3},-e_{1} + e_{2}, e_{3}) $,
\end{center}
 which gives the  Hilbert basis associated to monoid $ \sigma \cap \mathbb{Z}^{3}$.
\begin{center}
$ \mathrm{H}_{\sigma^{\vee}} = \lbrace e_{1} + e_{2},e_{1} + e_{2} + e_{3}, e_{1} + 2e_{2} \rbrace $.
\end{center}
From Theorem \ref{Theorem 2.3.0.}, one obtains  the geometric realization of the affine toric variety $ X_{\sigma'^{\vee}} $,\\
\begin{center}
$X_{\sigma'^{\vee}} $ = \textbf{Spec}$(\mathbb{C}[S_{\sigma'^{\vee}} \cap \mathbb{Z}^{3}])= \textbf{Spec}(\mathbb{C}[u_{1}u_{2},u_{1}u_{2}u_{3},u_{1}u^{2}_{2}])$.
\end{center}

One chooses the set of generators of $ X_{\sigma'^{\vee}} $, which forms a uni modular matrix $ A=Columns((1,1,0)',(1,1,1)',(1,2,0)') $, and   parametrize this system by means of monomials of Laurent.\\

By  Theorem \ref{Theorem of Hironaka-Atiyah}, one  obtains  the resolution map, 
\begin{center}
$ g_{1}: (a,b,c) \longrightarrow (u_{1}u_{2},u_{1}u_{2}u_{3},u_{1}u^{2}_{2})/(0,0,0) $, 
\end{center}
such that,\\
 \begin{center}
$ H(g_{1}(u_{1},u_{2},u_{3}))= u^{4}_{1}u^{4}_{2}u^{2}_{3} + 2u^{3}_{1}u^{4}_{2}u_{3} + u^{2}_{1}u^{4}_{2} + 3u^{6}_{1}u^{6}_{2}u^{4}_{3}$\\
$\qquad \quad\quad\quad\quad\quad=u^{2}_{1}u^{4}_{2}(u^{2}_{1}u^{2}_{3} + 2u_{1}u_{3} + 1 + 3u^{4}_{1}u^{2}_{2}u^{4}_{3}) $\\

$\quad\quad\quad\quad\quad=u^{2}_{1}u^{4}_{2}((u_{1}u_{3} + 1)^{2} + 3u^{4}_{1}u^{4}_{3}u^{2}_{2}) $\\

$\qquad =c^{2}_{1}((b_{1} + 1)^{2} + 3b^{4}_{1}d^{2}_{1}) $\\

$ \quad\quad=c^{2}_{1}(b'^{2}_{1} + 3(b'_{1} - 1)^{4}d^{2}_{1}) $\\

$=c^{2}_{1}(b'^{2}_{1} + 3e^{4}_{1}d^{2}_{1}) $.

\end{center}
By applying a second time the technique of resolution by means of Hilbert basis to the polynomial defined by $ h(b'_{1},d_{1},e_{1})=b'^{2}_{1} + 3e^{4}_{1}d^{2}_{1} $,  where  the basis Hilbert associated to Newton polytope of $ supp(h) $, in this manner, one has  that $ \mathrm{H_{\sigma^{\vee}}}=\lbrace e_{1},e_{1} + e_{2},e_{1} + 2e_{2} \rbrace $, and as consequence of the resolution map,
\begin{center}
$g_{2}:(b'_{1},d_{1},e_{1}) \longrightarrow (s_{1},s_{1}s_{2},s_{1}s^{2}_{2})/(0,0,0)$.
\end{center}
Then, one obtains the affine toric variety; $ X_{\sigma^{\vee}}= Spec\mathbb{C}[s_{1},s_{1}s_{2},s_{1}s^{2}_{2}] $, and so,
\begin{center}
$ H(g_{1}(g_{2}(s_{1},s_{2},s_{3})))=c^{2}_{1}s^{2}_{1}(1+ 3s^{4}_{1}s^{8}_{2}s^{2}_{2}) $\\
 $ =c^{2}_{1}s^{2}_{1}(1 + 3s^{4}_{1}s^{10}_{2}) $,
\end{center}
which is not singular in $ (0,0,0) \in \mathbb{R}^{3} $. This same fact is proved in, S. Watanabe, see, ref. \cite{SWatanabe42001}, \cite{SWatanabe62001}.\\

\textbf{Application (B). Mix of binomial distributions.} This kind of statistical learning is used by the spectral analysis of mutations, see \cite{MAoyaguiWatanabe2005}, and the  statistical machine is cha\-rac\-te\-ri\-zed by the following probabilities: \\

True probability of $ x $, $ q(x=k)=Bin_{N} (x;p^{*})= \binom {N}{x}p^{*x} (1 - p^{*})^{N - x}$.\\

True probabilistic inference of model $ P(x=k \vert w)= aBin_{N}(x,p_{1}) + (1 - a) Bin_{N}(x,p_{2}) $.\\

Parameter space $ w$ is defined by: 
\begin{center}
$ w = (\lbrace  a_{i}  \rbrace^{K}_{i=1}, \lbrace p_{i} \rbrace^{K+1}_{i=1}) $,
\end{center}
where  coordinates of the parameters $ p_{i} $ are defined in the range $ 0 < p_{i} <1/2 $ and,
\begin{center}
$ a_{K+1}= 1 - \sum^{K}_{i=1} a_{i} $.
\end{center}
There holds  the following theorem proved by Watanabe, see ref., \cite{KYamazakiWatanabe2004}, immediately, the same result is proven but using Hilbert basis and toric morphism: \\
 
\begin{theorem}\label{Theorem} Consider a learning machine characterized by the probabilities defined above, then for a number large enough $ n $ of training examples, in accordance with Corollary \ref{Corollary 4.3.1}, its learning curve is given by: 
\begin{center}
$ K(n)= \dfrac{3}{4} \log(n) + C $,
\end{center}
where $C$ is independent of $ n $.\\
\end{theorem}
\textbf{Proof.} The  Kullback information distance is given by: 

\begin{center}
$ \mathrm{H}(x,a,b_{1},b_{2})= \sum^{N}_{x=0} q(x) log\left( \dfrac{q(x)}{P(x \vert \omega)} \right) $
$ =(ap_{1} + (1 - a)p_{2})^{2} + (ap^{2}_{1} + (1 - a)p^{2}_{2})^{2} +....+$\\
$ =b^{2}_{2} + (ab^{2}_{1} + (b_{2} - ab_{1})^{2})^{2} +...+$ major order terms.
\end{center}

which is singular in $ (0,0,0) \in \mathbb{R}^{3}$. According to Theorem  \ref{Theorem of Hironaka-Atiyah}  and Hilbert basis lemma, and to our technique with toric morphisms, one can see that the previous  polynomial is generated by the ideal  $ \mathrm{I}< \mathbb{C}[a,b_{1},b_{2}] $\\
\begin{center}
$ \mathrm{I}:= < b_{2}^{2}; ab_{1}^{2}; ab_{1}b_{2}> $
\end{center}
and the lattice cone
\begin{center}
$ \sigma = Con((0,0,2);(1,2,0);(1,1,1)). $
\end{center}
Computing  the  geometric realization of the affine toric variety $ X_{\sigma^{\vee}} $, associated with monoid $ S_{\sigma}= \sigma^{\vee} \cap \mathbb{Z}^{3} $ and with Hilbert basis,  one obtains \\
\begin{center}
$ \mathrm{H}_{\sigma^{\vee}}= \lbrace e_{3}; e_{1} + e_{2} + e_{3}; e_{1} + 2e_{2}  \rbrace $,
\end{center}
where the toric variety is: 
\begin{center}
$ X_{\sigma^{\vee}} =$\textbf{Spec}$\mathbb{C}[w_{3},w_{1}w_{2}w_{3},w_{1}w_{2}^{2}] $,
\end{center}
and the coordinate system: 
\begin{center}
$a= w_{3} ; $\\
$ b_{1}=w_{1}w_{2}w_{3}; $\\
$ b_{2}=w_{1}w_{2}^{2}$,
\end{center}
and  using  this parametrization we get the resolution map $ g : X_{\sum'} \longrightarrow X_{\sum} $  such that  $ \mathrm{H}(g(w)), w=(w_{1},w_{2},w_{3}) \in \mathbb{R}^{3}$ is not singular in $ (0,0,0) $ and from Theorem \ref{Theorem of Hironaka-Atiyah};

$ \mathrm{H}(g(w))= w_{1}^{2}w_{2}^{4} + (w_{3}w_{1}^{2}w_{2}^{2}w_{3}^{2}+ (w_{1}w_{2}^{2} - w_{1}w_{2}w_{3}^{2})^{2})^{2} +...+$ order major terms),\\
$=w_{1}^{2}w_{2}^{4} + [w_{1}^{2}w_{2}^{2} + (w_{3}^{3}+(w_{2} -w_{3}^{2})^{2})]^{2} +...+$ order major terms,\\
$ =w_{1}^{2}w_{2}^{4} + w_{1}^{4}w_{2}^{4}[w_{3}^{3}(w_{2}-w_{3}^{2})^{2}]^{2} +...+$ order major terms.\\
$ =w_{1}^{2}w_{2}^{4}(1 + w_{1}^{2}(w_{3}^{6} +2w_{3}^{3}(w_{2}- w_{3}^{2})^{2}+(w_{2}-w_{3})^{4} +...+$ order major terms).

It is easy to see that writing the terms of integration  of $ J(z) $,  we get, 
\begin{center}
$ \displaystyle J(z)=\int H(g(w))^{z} \vert g'(w) \vert du$
$  \displaystyle  =\int ((1 + w_{3}^{2}w_{1}^{2}+...)w_{2}^{4}w_{1}^{2})^{z} \vert w_{2}^{2}w_{1} \vert dw_{1}dw_{2}dw_{3} $
$ =\dfrac{f(z)}{4z + 3} $,
\end{center}
where the most large pole of $ J(z) $ is $ \lambda_{1}=\dfrac{3}{4} $ and multiplicity $ m_{1}=1 $, then the learning curve is given by: \\
\begin{center}
$ K(n)=\dfrac{3}{4} \log(n) + C $;   
\end{center}
\textbf{q.e.d.} see ref. \cite{KYamazakiWatanabe2004}.\\

\textbf{Application (C).} The following application is the toric resolution in a perceptron of three layer on the learning curve of the same, for details of the computing of this learning curve, see, Watanabe \cite{SWatanabe42001}. We define the machine in the space probability $ (\Omega, \mathbb{F}, P) $ as follows: \\
\begin{enumerate}
\item \textit{A priori} probability distribution, $ \varphi(\omega) > 0 $.
\item Predictive probability of the vector $ y \in \mathbb{R}^{N} $, 
\begin{center}
$P(y \vert x,\omega)=\dfrac{1}{(2\pi s^{2})^{N/2}}\exp \left( \dfrac{-1}{2s^{2}} \Vert y - f_{k}(x,\omega)\Vert^{2} \right)$, 
\end{center}
with $ x \in \mathbb{R}^{M} $ and $ s > 0 $ is the standard deviation.
\item True probability distribution of model, 
\begin{center}
$ q(y \vert x)q(x) = \dfrac{1}{(2\pi s^{2})^{N/2}}\exp \left( -\dfrac{1}{2s^{2}}\Vert y \Vert^{2}\right)  q(x)$. 
\end{center}
\end{enumerate}
First we compute the Kullback distance of this machine is defined as follows, 
\begin{center}
$ \displaystyle  H(a,b,c)= \dfrac{1}{2s^{2}}\int \Vert f_{K}(x,a,b,c) \Vert^{2} q(x) dx  $ 
 $ =\sum^{N}_{p=1} \sum^{K}_{h,k=1} B_{hk}(b,c) a_{hp} b_{kp} $.
\end{center}
with parameter space associated, and function of the hidden units given by, $ f_{K}(x,\omega)= \sum^{K}_{k=1} a_{k} \sigma(b_{k} x + c_{k}) $ {}  :
\begin{center}
$ a=\lbrace a_{k} \in \mathbb{R}^{N}; k=1,2...,K   \rbrace $\\
$ b= \lbrace b_{k} \in \mathbb{R}^{M}; k=1,2,...,K \rbrace $\\
$ c=\lbrace c_{k} \in \mathbb{R}; k=1,2,...,K \rbrace $,\\
$ a_{k}=\lbrace a_{kp} \in \mathbb{R}; p=1,2...,N    \rbrace $\\
$ b_{k}=\lbrace b_{kp} \in \mathbb{R};q=1,2,...,M \rbrace $.
\end{center}
defining $ \displaystyle   B_{hk}(b,c)=\dfrac{1}{2s^{2}}\int \sigma(b_{h}*x + c_{h})\sigma(b_{k}*x + c_{k})q(x)dx $, where $ \sigma(x)=tanh(x) $ is a synaptic function, see for details of the formulation, Watanabe \cite{SWatanabe42001}. Thus developing terms.
\begin{center}
$  H(a,b,c)=\sum_{p=1}^{N}(B_{11}(b,c)a_{1p}a_{1p} + B_{22}(b,c)a_{2p}a_{2p}+...+B_{KK}a_{Kp}a_{Kp})$\\
$ =B_{11}a_{11}^{2}+B_{22}a_{21}^{2}+....+B_{KK}a_{K1}^{2}$\\
$+B_{11}a_{12}^{2}+B_{22}a_{22}^{2}+....+B_{KK}a_{K2}^{2}  $\\
$ +B_{11}a_{13}^{2}+B_{22}a_{23}^{2}+....+B_{KK}a_{K3}^{2} $\\
$ +B_{11}a_{14}^{2}+B_{22}a_{24}^{2}+...+B_{KK}a_{K4}^{2}+ $\\
$ .............+......+......$\\
$ +B_{11}a_{1N}^{2}+B_{22}a_{2N}^{2}+....+B_{KK}a_{KN}^{2}. $
\end{center}
This polynomial seen in the coordinates $ a_{hk}\in\mathbb{R} $, is singular in $ (0,0,...,0)\in\mathbb{R} ^{N} $, we construct a toric resolution in this coordinates utilizing the concept of projective sets as previously we have seen. We define the affine charts utilizing the following projective set, 
\begin{center}
$ U_{j}=\lbrace[a_{11},...,a_{1N},a_{21},...,a_{2N},...,a_{K1},...,a_{KN}]\in \mathbb{R} P^{KN-1}:a_{jj}\neq 0\rbrace$.
\end{center}
Where $ \mathbb{R} P^{KN-1}$ is the real projective space $ KN-1 $-dimensional, and also there exists a bijection as we have seen, the affine real space $ KN-1 $-dimensional  $ \mathbb{R}^{KN-1} $, given by, 
\begin{center}
$ U_{j}:\mathbb{R} P^{KN-1}\longmapsto \mathbb{R} ^{KN-1}   $
\end{center}
\begin{center}
$  U_{j}:[a_{11},...,a_{1N},a_{21},...,a_{2N},...,a_{K1},...,a_{KN}]\longmapsto
(1,a_{12}a_{11}^{-1},...,a_{1N}a_{11}^{-1},a_{21}a_{11}^{-1},...,a_{2N}a_{11}^{-1},...
,a_{K1}a_{11}^{-1},...,a_{KN}a_{11}^{-1})$
\end{center}
now in these projective coordinates we redefine $ H(a,b,c)=u_{11}^{2}H_{1}(a,b,c) $, since:
\begin{center}
$ H(a,b,c)=a_{12}^{2}u_{11}^{2}B_{11}+a_{13}^{2}u_{11}^{2}B_{11}+...+a_{1N}^{2}u_{11}^{2}B_{11}+a_{21}^{2}u_{11}^{2}B_{22} +...+ a_{2N}^{2}u_{11}^{2}B_{22}+...+
a_{k1}^{2}u_{11}^{2}B_{KK}+...+a_{KN}^{2}u_{11}^{2}B_{KK}$
\end{center}
where the new coordinates in the affine space $ \mathbb{R} ^{KN-1} $ son $ (u_{11},a_{12},...,a_{1N},...,a_{K1},...,a_{KN})\\
\in \mathbb{R} ^{KN-1} $, in this new coordinate ring  we construct the lattice cone of Newton polytope associated in the re parametrized polynomial, which give us as: $ \sigma=Con(2e_{1}+2e_{2},...,2e_{1}+2e_{1N},...,
2e_{1}+2e_{K1},...,2e_{1}+2e_{KN}) $, that in matrix way give rises the following array associated to the cone,
\\
\\
\\\\
$ A_{\sigma}= $
\begin{center}
\begin{tabular}{c c c c c c c}
2&2&0&.&.&.&0 \\
2&0&2&.&.&.&0\\
2&0&0&2&.&.&0\\
.&.&.&.&.&.&.\\
.&0&.&.&.&2&0\\
2&0&.&.&.&0&2\\\\\\
\end{tabular}
\end{center}

It possible to show in a inductive way, and using of Singular program, \cite{DGPS}, that the Hilbert basis associated to this lattice cone, are given by the following matrix array,\\\\

$ H_{\sigma^{\vee}}= $
\begin{center}
\begin{tabular}{c c c c c c c}
1&0&0&.&.&.&1 \\

1&0&.&.&.&1&0\\

1&0&.&.&1&0&0\\

.&.&.&.&.&.&.\\
.&.&1&0&.&.&0\\
1&1&0&.&.&.&0\\\\

\end{tabular}
\end{center}
where the lattice vectors of this array represent us a regular lattice cone; and by The Theorem \ref{Theorem 2.3.0.}  and Theorem \ref{Theorem of toric modification}, we have a toric blow up or toric resolution, also we obtain the respectively toric variety $ X_{\sigma^{\vee}} $, taking as exponents the elements of this base for the constructing the monomial homomorphisms ( Theorem \ref{Theorem 2.3.0.} ), and so the following transformation of monomial coordinates:
\begin{center}
$ a_{11}=v_{11}; $\\
$  u_{11}=v_{11}^{-1}$,\\
$ a_{hp}=u_{11}*u_{hp}$; $ \forall $  $h\neq 1 $ \'{o} $p\neq 1$.

\end{center}
which is the re parametrization shown in, Watanabe \cite{SWatanabe42001}. Furthermore we set a extra coordinate $ u_{11} $, being that we work with the projective set $ U_{j} $, where we construct from the affine chart;
\begin{center}
$ A_{0}=\lbrace(1,a_{12}a_{11}^{-1},...,a_{1N}a_{11}^{-1},a_{21}a_{11}^{-1},...,a_{2N}a_{11}^{-1},...
,a_{K1}a_{11}^{-1},...,a_{KN}a_{11}^{-1}) \vert a_{11} \neq 0 \in\mathbb{R}^{KN-1} \rbrace $, 
\end{center}
so explicitly we have the toric variety as the affine algebraic scheme 
\begin{center}
 $ X_{\sigma^{\vee}} = $\textbf{Spec}$(\mathbb{C}[\sigma^{\vee} \cap \mathbb{Z}^{KN-1}])=A_{0}  $;
\end{center} 
is enough to realize the toric resolution in this chart being that the toric morphism are proper, and they extend at all the variety. Finally by the Watanabe's theorems, computing the largest pole of zeta function,
\begin{center}
$\displaystyle  J(z)=\int_{U(\delta)} H(g(u),b,c)^{z} \varphi_{0} \vert g'(u) \lambda du'dbdc $.\\
\end{center}
In Watanabe is shown that this toric resolution is not complete and is necessary other resolution to the Kullback distance $ H(g(u),b,c) $ applying Hilbert basis again, now the monomial transformation is given by, 
\begin{center}
$ g :\lbrace u_{kp},v_{k}; 1\leqslant k \leqslant K; 1 \leqslant p \leqslant M  \rbrace  \mapsto \lbrace b_{kp},c_{k}; 1 \leqslant k \leqslant K; 1 \leqslant p \leqslant M \rbrace$.
\end{center}
that is defined by,
\begin{center}
$ b_{11}=u_{11} $\\
$\qquad \qquad \qquad \qquad  b_{kp}=u_{11}u_{kp}, \quad (k\neq 1)$ o $ (p\neq 1) $,\\
$ c_{k}= u_{11}v_{k} $.
\end{center}
Then by Atiyah-Hironaka theorem exists analytic function $ H_{2}(a,u',v) $, such that,
\begin{center}
$ H(a,b,c)= u^{2}_{11}H(a,u',v) $,
\end{center}
which implies;
therefore $ \lambda_{1} \leqslant (M + 1)K/2 $.\\
Combining the results of above, the largest pole $ -\lambda_{1} $ of the poles of $ J(z)$ satisfies of inequality,
\begin{center}
$ \lambda_{1} \leqslant \dfrac{K}{2} min\lbrace N,M +1 \rbrace $,
\end{center}
With this information and corollary  \ref{Corollary 4.3.1}, we have the learning curve associated to the perceptron:\\
\begin{center}
$ K(n) \leqslant \dfrac{K}{2} min\lbrace N,M+1  \rbrace \log(n) + o\left( \log (n)\right)  $.
\end{center}
We reproduce the first toric resolution in the first re parametrization of the Kulback  distance and its toric variety associated by means of Hilbert basis; and the second resolution is given in Watanabe \cite{SWatanabe42001}; but is possible apply the technique with Hilbert basis the necessary times to up having the wished resolution, agreement to the Hironaka's theorem, \cite{HHironaka1964}.

\thispagestyle{plain}
\section*{\begin{center}
Conclusions.
\end{center}}
 The principal conclusion of this work is the use of Theorem \ref{Theorem 2.3.0.}; as consequence, Theorem \ref{Corollary 4.3.3}, that are fundamentals for the formalization and for reproducing of results previously reported in statistical singular learning, S. Watanabe \cite{SWatanabe12001}, \cite{SWatanabe42001}, \cite{SWatanabe52001}. The practical applications of these theorems is obtained by means of the use of Hilbert basis with Singular program, ref., \cite{DGPS}. This open the doors to look for other perspectives of investigation for machines with a high dimensional parameter space important in data science.It should be clear that the algorithmic complexity for computing Hilbert basis is a topic of current interest in computational algebraic geometry, but its solution for lattice polytope with thousands of vertex may well help in the solution of many other problems beyond the examples presented here.\\\\

\end{document}